\documentclass[fleq]{amsart}
\usepackage{amssymb,amsmath,latexsym,amsbsy,color,enumerate}
\numberwithin{equation}{section}

\newtheorem{theorem}{Theorem}

\newtheorem{remark}{Remark}

\begin{document}
 \title[Curves with decomposable normal vector bundles]{Curves with decomposable normal vector bundles and automorphism groups}
 \author{Tuyen Trung Truong}
    \address{School of Mathematics, Korea Institute for Advanced Study, Seoul 130-722, Republic of Korea}
 \email{truong@kias.re.kr}
\thanks{}
    \date{\today}
    \begin{abstract}If a smooth projective threefold $X$ satisfies a certain Property A (see below for definition), then any automorphism of $X$ has zero entropy. Let  $Y$ be a smooth projective threefold satisfying Property A. Let $\pi :X\rightarrow Y$ be a blowup at either a point or at a smooth curve $C\subset Y$ with the following two properties: i) $c_1(Y).C$ is an odd number, and ii) the normal vector bundle $N_{C/Y} $ is decomposable. Then we show that $X$ also satisfies Property A. 

As a further application of Property A we prove the following result. Let $X_1$ be the blowup of $X_0=\mathbb{P}^3$ at a finite number of points, and let $X=X_2$ be the blowup of $X_1$ at a finite number of pairwise disjoint smooth curves (here the images of these curves in $X_0$ may intersect). Then any automorphism of $X$ has the same first and second dynamical degrees. Under some further conditions, then any automorphism of $X$ has zero entropy. The result is also valid for threefolds $X_0$ satisfying a certain condition on the second Chern class. Some explicit examples are given.
\end{abstract}
\maketitle
\section{Introduction}
It is very difficult to find automorphisms of positive entropy on a smooth rational threefold. In fact, the following question (asked in 2011) by E. Bedford: "Is there a projective threefold which is a finite composition of blowups at points or smooth curves starting from $\mathbb{P}^3$ and which has an automorphism of positive entropy?" still has no answer.  Even for the larger class of smooth rational threefolds, there are currently only two known examples of manifolds with primitive automorphisms of positive entropy (see \cite{oguiso-truong, catanese-oguiso-truong, colliot-thelene}). Here  a primitive automorphism, defined by D.-Q. Zhang \cite{zhang}, is one that has no non-trivial invariant fibrations.   

In \cite{truong}, we gave many evidences to that the answer to Bedford's question is No. The work in \cite{truong} has been generalized to higher dimensions in \cite{bayraktar-cantat} and \cite{truong1}.   We note that in contrast, there are such blowups $X$ with interesting pseudo-automorphisms which are primitive (see \cite{bedford-kim, bedford-cantat-kim}). 

Let $X$ be a smooth projective threefold. We denote $H^{1,1}(X,\mathbb{R})=H^2(X,\mathbb{R})\cap H^{1,1}(X)$. A cohomology class $\zeta \in H^{1,1}(X,\mathbb{R})$ is nef if $\zeta $ is the limit of a sequence of K\"ahler forms on $X$. Let $K_X\in H^{1,1}(X,\mathbb{R})$ denote the canonical class of $X$, and $c_j(X)$ is the $j$-th Chern class of $X$. In \cite{truong}, we used that a threefold has no automorphism of positive entropy as long as it satisfies the following property.

{\bf Property A1.} A smooth projective threefold $X$ satisfies Property A1 if  whenever $\zeta $ is a nef cohomology class on $X$ such that $\zeta ^2=0$, then $\zeta \in \mathbb{R}.H^2(X,\mathbb{Q})$. 

{\bf (Non-) Example 1. } However, it can be checked that there are some finite composition of smooth blowups $X\rightarrow \mathbb{P}^3$ starting from $\mathbb{P}^3$ for which Property A1 above is violated. For an explicit example we can proceed as follows. There is a finite composition of smooth blowups $Z\rightarrow \mathbb{P}^2\times \mathbb{P}^1$ starting from $\mathbb{P}^2\times \mathbb{P}^1$ with an automorphism $f$ of positive entropy (these can be constructed from automorphisms of positive entropy on some finite composition of smooth blowups starting from $\mathbb{P}^2$, for example those given in \cite{mcmullen} and \cite{bedford-kim1}). Then, there is a  non-zero nef $\xi $ such that $\xi ^2 =0$  and $\xi \notin \mathbb{R}.H^2(Z,\mathbb{Q})$. Since $\mathbb{P}^3$ and $Z$ are birationally equivalent, by Hironaka's resolution of singularities there is a finite composition of smooth blowups $X\rightarrow \mathbb{P}^3$ starting from $\mathbb{P}^3$, which has a surjective birational morphism $\pi :X\rightarrow Z$. Then $\zeta =\pi ^*(\xi )$ is nef on $X$, $\zeta ^2=0$ , but $\zeta\notin \mathbb{R}.H^2(X,\mathbb{Q})$.

Therefore, we see that Property A1 can not be used to check for a general finite composition of smooth blowups starting from $\mathbb{P}^3$. We note that if $f:X\rightarrow X$ is a holomorphic automorphism and $T_X$ is the holomorphic tangent bundle of $X$, then the differential map $f_*$ defines an isomorphism between $T_X$ and $f^*T_X$. In particular, the Chern classes of $X$ are preserved by  $f$.  Based on this, we propose an improved condition. 

{\bf Property A.} A smooth projective threefold $X$ satisfies Property A if  whenever $\zeta $ is a nef cohomology class on $X$ such that $\zeta ^2=0$, $\zeta .c_1(X)^2\geq 0$ and $\zeta .c_2(X)\leq 0$, then $\zeta \in \mathbb{R}.H^2(X,\mathbb{Q})$. 

It can be shown similarly to \cite{truong} that if a smooth projective threefold $X$ satisfies Property A then any automorphism on $X$ has zero entropy. If a smooth projective threefold $X$ satisfies Property A1, then obviously it satisfies Property A. Return to the example above, if $\xi$ on $Z$ is nef such that $\xi ^2=0$, $\xi .c_1(Z)^2\geq 0$ and $\xi .c_2(Z)\leq 0$, then $\zeta =\pi ^*(\xi )$ on $X$ is still nef and satisfies $\zeta ^2=0$. However, the conditions $\zeta .c_1(X)^2\geq 0$ and $\zeta .c_2(X)\leq 0$ are not guaranteed. 

It is more natural that in Property A, instead of the conditions $\zeta .c_1(X)^2\geq 0$ and $\zeta .c_2(X)\leq 0$ we should put the stronger conditions $\zeta .c_1(X)^2=0$ and $\zeta .c_2(X)=0$. However, the conditions $\zeta .c_1(X)^2\geq 0$ and $\zeta .c_2(X)\leq 0$ behave well under a blowdown, which is good for inductive arguments (see part 1) of the proof of Theorem \ref{Theorem1} below), while this is not the case for the conditions $\zeta .c_1(X)^2=0$ and $\zeta .c_2(X)=0$.
 
 Here is the first main result of this paper. 

\begin{theorem}
Let $Y$ be a smooth projective threefold satisfying Property A. Let $\pi :X\rightarrow Y$ be a blowup at a point, or at a smooth curve $C$ satisfying the following two conditions: 

i) $c_1(Y).C$ is an odd number, 

and 

ii)  The normal vector bundle $N_{C/Y}$ is decomposable, i.e. it is a direct sum of two line bundles over $C$. 

Then $X$ also satisfies Property A. 

\label{Theorem1}\end{theorem}
 \begin{remark}We note that the condition ii) in Theorem \ref{Theorem1} may be easily satisfied. For example, if $C$ is a smooth rational curve in $Y$, then even if $C$ does not move in $Y$, its normal vector bundle in $Y$ is still decomposable, by a result of Grothendieck. 
\label{Remark1}\end{remark}

Our next main result is a further application of Property A. It roughly says that for a blowup $X$ of $\mathbb{P}^3$ at a finite number of curves in $\mathbb{P}^3$ which may intersect each other, any automorphism of $X$ has the same first and second dynamical degrees. If some additional assumptions are imposed, then any automorphism of $X$ has zero entropy.  The result is also valid for more general $X_0$. The precise statement  will be stated after we recall some basic notions.  

If $f:X\rightarrow X$ is an automorphism of a smooth projective threefold, then the first dynamical degree $\lambda _1(f)$ is defined as the largest eigenvalue of $f^*:H^{1,1}(X)\rightarrow H^{1,1}(X)$.  We then define $\lambda _2(f):=\lambda _1(f^{-1})$. By Gromov-Yomdin's theorem, the entropy of $f$ equals $\log \max \{\lambda _1(f),\lambda _2(f)\}$. These dynamical degrees satisfy a log-concavity property: $\lambda _1(f)^2\geq \lambda _2(f)$. We note that if $f$ preserves a fibration over a curve or a surface, then its first and second dynamical degrees are the same.

Let $H^{1,1}_{alg}(X,\mathbb{Q})\subset H^{1,1}(X,\mathbb{Q})$ denote the subvector space generated by the classes of divisors in $X$. We define $H^{1,1}_{alg}(X,\mathbb{R})=\mathbb{R}\otimes _{\mathbb{Z}}H^{1,1}_{alg}(X,\mathbb{Q})$. Then $\lambda _1(f)$ is the same as the largest eigenvalue of $f^*:H^{1,1}_{alg}(X)\rightarrow H^{1,1}_{lag}(X)$. An element $\zeta \in H^{1,1}_{alg}(X,\mathbb{R})$ is nef if it is the limit of ample divisors with real coefficients. An element $\zeta\in H^{1,1}_{alg}(X,\mathbb{R})$ is movable if there is a blowup $\pi :Z\rightarrow X$ such that $\zeta$ is the pushforward of some nef class on $Z$. 

\begin{theorem}
Let $X_0$ be a smooth projective threefold such that $c_2(X_0).\zeta >0$ for all non-zero movable $\zeta \in H^{1,1}_{alg}(X_0,\mathbb{R})$. Let $X_1\rightarrow X_0$ be the blowup at a finite number of points in $X_0$. Let $D_1,\ldots ,D_m\subset X_1$ be pairwise disjoint smooth curves, and $X=X_2\rightarrow X_1$ the blowup at these curves. Let $f$ be an automorphism of $X$. Then  

1) $\lambda _1(f)=\lambda _2(f)$.

2) Assume moreover that for any $j$, then $c_1(X_1).D_j\leq 2g_j-2$, where $g_j$ is the genus of $D_j$. Then any automorphism of $X$ has zero entropy.
\label{Theorem2}\end{theorem}

To prove part 1) of the theorem, we use the following analog of Condition A: If $\zeta \in H^{1,1}_{alg}(X)$ is nef and is not contained in $\mathbb{R}.H^{1,1}_{alg}(X,\mathbb{Q})$, then either $\zeta ^2\not=0$ or $\zeta .c_1(X)\not= 0$ or $\zeta .c_2(X)\not= 0$. The only difference is that here we require a weaker condition $\zeta .c_1(X)\not= 0$, while in Condition A we require a stronger one $\zeta .c_1(X)^2\not= 0$. 

\begin{remark}Let $X_0$ be a smooth projective threefold which is a complete intersection in $\mathbb{P}^n$, where $n\geq 4$. That is, $X_0$ is the intersection of $n-3$ hypersurfaces $V_1,\ldots ,V_{n-3}$. We now show that if $\zeta $ is a non-zero movable class on $X_0$, then $\zeta .c_2(X_0)>0$.

Let $d_1,\ldots ,d_{n-3}$ be the degrees of $V_1,\ldots ,V_{n-3}$.  Let $h$ be the class of a hyperplane on $X$. The Chern classes of the normal bundle $N_{X_0/\mathbb{P}^n}$ is given by the formula
\begin{eqnarray*}
c(N_{X_0/\mathbb{P}^n})=\prod _{j=1}^{n-3}(1+d_jh).
\end{eqnarray*}
In particular,
\begin{eqnarray*}
c_1(N_{X_0/\mathbb{P}^n})&=&(\sum _{j}d_j)h,\\
c_2(N_{X_0/\mathbb{P}^n})&=&(\sum _{i<j}d_id_j)h^2.
\end{eqnarray*}
From the exact sequence
\begin{eqnarray*}
0\rightarrow T_{X_0}\rightarrow T_{\mathbb{P}^4}|_{X_0}\rightarrow N_{X_0/\mathbb{P}^3}\rightarrow 0,
\end{eqnarray*}
and the splitting principle for Chern classes, it follows that
\begin{eqnarray*}
c_1(X_0)&=&c_1(\mathbb{P}^n)|_{X_0}-c_1(N_{X_0/\mathbb{P}^n})= ((n+1)-\sum _jd_j)h,\\
c_2(X_0)&=&c_2(\mathbb{P}^n)|_{X_0}-c_2(N_{X_0/\mathbb{P}^n})-c_1(X_0)c_1(N_{X_0/\mathbb{P}^n})\\
&=&(\frac{(n+1)n}{2}-\sum _{i<j}d_id_j-(n+1)\sum _{j}d_j+(\sum _{j}d_j)^2)h^2.
\end{eqnarray*}
We have
\begin{eqnarray*}
&&\frac{(n+1)n}{2}-\sum _{i<j}d_id_j-(n+1)\sum _{j}d_j+(\sum _{j}d_j)^2\\
&=&[\frac{n-4}{2(n-3)}(\sum _jd_j)^2-\sum _{i<j}d_id_j]+[\frac{n(n+1)}{2}+\frac{n-2}{2(n-3)}(\sum _jd_j)^2-(n+1)\sum _jd_j].
\end{eqnarray*}
By Cauchy-Schwarz inequality, the first bracket on the right hand side of the above expression is non-negative. We now show that the second bracket is positive. We define $x=\sum _jd_j$. Then $x$ is a positive integer which is $\geq n-3$, and the second bracket is quadratic in $x$:
\begin{eqnarray*}
\frac{n(n+1)}{2}+\frac{n-2}{2(n-3)}(\sum _jd_j)^2-(n+1)\sum _jd_j=\frac{n(n+1)}{2}-(n+1)x+\frac{(n-2)}{2(n-3)}x^2=: g(x).
\end{eqnarray*}
The critical point of $g$ is $x_0=(n+1)(n-3)/(n-2)<n$. Hence, to show that $g(x)>0$ for all positive integer $x\geq n-3$, it suffices to show that $g(n-3),g(n-2),g(n-1),g(n)>0$ for any positive integer $n\geq 4$. We now check this latter claim. 

For $x=n-3$
\begin{eqnarray*}
g(n-3)=6>0.
\end{eqnarray*}
(Note that in this case all $d_j$ are $1$ and $X_0$ is not different than $\mathbb{P}^3$.)

For $x=n-2$, using that $(n-2)^2> (n-1)(n-3)$, we obtain
\begin{eqnarray*}
g(n-2)>\frac{n(n+1)}{2}-(n^2-n-2)+\frac{(n-2)(n-1)}{2}=3>0.
\end{eqnarray*}

For $x=n-1$, we have
\begin{eqnarray*}
g(n-1)=\frac{2(n-2)}{(n-3)}>0.
\end{eqnarray*}

For $x=n$, we have
\begin{eqnarray*}
g(n)=\frac{1}{(n-3)}>0.
\end{eqnarray*}

 A movable class is in particular  psef, i.e. can be represented by a positive closed current. Hence,  if $\zeta$ is a non-zero movable class on $X_0$ then $\xi .c_2(X_0)>0$.  Hence, Theorem \ref{Theorem2} can be applied for such a $X_0$. 
\end{remark}

The last main result in this paper considers some cases not covered in Theorem \ref{Theorem2} and is specific for the case where $X_0=\mathbb{P}^3$.
\begin{theorem}

1) Let $p_1,\ldots ,p_n$ be distinct points in $X_0=\mathbb{P}^3$, any $4$ of them do not belong to the same plane.  Let $C_{i,j}\subset \mathbb{P}^3$ be the line connecting $p_i$ and $p_j$. Let $\pi _1:X_1\rightarrow X_0$ be the blowup at $p_1,\ldots ,p_n$, and let $D_{i,j}\subset X_1$ be the strict transforms of $C_{i,j}$. Let $\pi _2:X=X_2\rightarrow X_1$ be the blowup of $X_1$ at all curves  $D_{i,j}$. Then any automorphism of $X$ has zero entropy.    
\label{Theorem3}\end{theorem}

\begin{remark}  
1) Igor Dolgachev informed us that in the special case of Theorem \ref{Theorem3} when $n=4$ or $5$, and $D_j$ are lines in $\mathbb{P}^3$, then the automorphism group of $X$ can be explicitly determined by methods which are different from ours.  

2) The proof of Theorem \ref{Theorem3} shows that the conclusion is still valid in the following more general setting. Let $\pi _1:X_1\rightarrow X_0=\mathbb{P}^3$ be the blowup at $n$ points $p_1,\ldots ,p_n$. Let $E_1,\ldots ,E_n$ be the exceptional divisors. Let $D_1,\ldots ,D_m\subset X_1$ be pairwise disjoint smooth curves. Let $X=X_2$ be the blowup of $X_1$ at $D_1,\ldots ,D_m$. We define
\begin{eqnarray*}
\gamma :=\sum _j\deg (\pi _1)_*(D_j).
\end{eqnarray*} 
Assume that there is $\lambda >0$ such that for any $l$:
\begin{eqnarray*}
\sum _{j}E_l.D_j\leq \lambda ,
\end{eqnarray*}
and moreover
\begin{eqnarray*}
\frac{6+\gamma }{\lambda}>\frac{11}{2}.
\end{eqnarray*}
Moreover, assume that for any $j$ 
\begin{eqnarray*}
(\frac{1}{2}+\frac{1}{\lambda})c_1(X_1).D_j\geq \frac{g_j-1}{2},
\end{eqnarray*}
where $g_j$ is the genus of $D_j$.
\end{remark}

Finally, we give some explicit examples.

{\bf Example 3}. We let $C_1,\ldots ,C_m\subset \mathbb{P}^3$ be pairwise disjoint smooth curves. Let $\pi _1:X_1\rightarrow \mathbb{P}^3$ be the blowup of $\mathbb{P}^3$ at these curves. Now let $D$ be a smooth rational curve on $X_1$. Assume that one of the following properties are satisfied: 

either

i) $D$ is a fiber of  an exceptional divisor in $X_1$, 

or 

ii) $\pi _1(D)$ intersects $\bigcup C_j$, counted with multiplicities, at an odd number of points. 

Let $X_2$ be the blowup of $X_1$ at $D$. Then $X_2$ satisfies Property A. We can iterate this procedure to obtain many more examples.

\begin{proof}[Proof of Example 3]
It follows from  the proof of Theorem 2 in \cite{truong} that $X_1$ satisfies Property A. 

Since $D$ is a smooth rational curve, condition 2) in Theorem \ref{Theorem1} is satisfied by Grothendieck's theorem. 

Hence, it suffices to check condition i) in Theorem \ref{Theorem1}. 

If $D$ is a fiber of an exceptional divisor $F$ in $X_1$, then $F.D=-1$ since $D$ is a fiber of $F$. Then, from 
\begin{eqnarray*}
c_1(X_1).D=2-2g+F.D=1,
\end{eqnarray*}
we have that $c_1(X_1).D$ is an odd number. 

Let $m$ be the number of point intersections of $\pi _1(D)$ and $\bigcup C_j$. Then, since $D$ is the strict transform of $\pi _1(D)$, we find that
\begin{eqnarray*}
c_1(X_1).D=c_1(\mathbb{P}^3).\pi _1(D)-m=4\deg (\pi _1(D))-m, 
\end{eqnarray*}
which is an odd number since $m$ is an odd number. 

Hence, in both cases Theorem \ref{Theorem1} applies.  

\end{proof}

{\bf Acknowledgments.} We would like to thank Igor Dolgachev for inspiring discussions and for his generous help on the topic.

\section{Proofs of the main results}

\begin{proof}[Proof of Theorem \ref{Theorem1}]
Let $\zeta$ be a nef class in $X$ such that both $\zeta .\zeta =0$, $\zeta .c_1(X)^2\geq 0$ and $\zeta .c_2(X)\leq 0$. We need to show that $\zeta \in \mathbb{R}.H^2(X,\mathbb{Q})$. 

Let $F$ be the exceptional divisor of the blowup $\pi$. We can write $\zeta =\pi ^*(\xi )-\alpha F$, for some $\xi \in H^{1,1}(Y,\mathbb{R})$ and $\alpha \geq 0$. We have several separate cases:

1) $\pi$ is a blowup at a point. In this case $c_1(X)=\pi ^*c_1(Y)-2F$ and $c_2(X)=\pi ^*c_2(Y)$ (see Chapter 4 in \cite{griffiths-harris}). We have 
\begin{eqnarray*}
\zeta ^2&=&(\pi ^*(\xi )^2-\alpha F)^2=\pi ^*(\xi ^2)+\alpha ^2F^2,\\
\zeta .c_1(X)^2&=&(\pi ^*(\xi )-\alpha F).(\pi ^*c_1(Y)-2F)^2=\pi ^*(\xi .c_1(Y)^2)-4\alpha .
\end{eqnarray*}
Here we use that $\pi _*(F)=\pi _*(F^2)=0$, and $F^3=1$. From $\zeta ^2=0$, it follows that $\alpha =0$. Thus $\zeta =\pi ^*(\xi )$, which implies that $\xi $ is nef and $\xi ^2=0$. Then the conditions $\zeta .c_1(X)^2\geq 0$ and $\zeta .c_2(X)\leq 0$ become (here $\alpha =0$) $\xi .c_1(Y)^2 \geq 0$ and $\zeta .c_2(Y)\leq 0$. Since $Y$ satisfies Property A, we have that $\xi\in \mathbb{R}.H^2(Y,\mathbb{Q})$. Consequently, $\zeta =\pi ^*(\xi )\in \mathbb{R}.H^2(X,\mathbb{Q})$.

2) $\pi $ is the blowup at a smooth curve $C$. In this case $c_1(X)=\pi ^*c_1(Y)-F$ and $c_2(X)=\pi ^*c_2(Y)+\pi ^*(C)-\pi ^*c_1(Y).F$ (see Chapter 4 in \cite{griffiths-harris}). We have two subcases:

{\bf Subcase 2.1}: $\alpha =0$. In this case $\zeta =\pi ^*(\xi )$. Since $\zeta$ is nef and $\zeta ^2=0$,  we have $\xi$ is nef and $\xi ^2=0$.  We have
 \begin{eqnarray*}
 0\leq \zeta .c_1(X)^2&=&\pi ^*(\xi )(\pi ^*c_1(Y)-F).(\pi ^*c_1(Y)-F)=\pi ^*(\xi .c_1(Y)^2)+\pi ^*(\xi ).F^2\\
 &=&\xi .c_1(Y)^2-\xi .C.
 \end{eqnarray*} 
Here we use that $\pi _*(F)=0$ and $\pi _*(F^2)=-C$ (see for example Lemma 4 in \cite{truong}). Hence, 
$$\xi .c_1(Y)^2\geq \xi .C\geq 0,$$ 
the last inequality follows from the fact that $\xi $ is nef and $C$ is an effective curve. 

Similarly, we have
\begin{eqnarray*}
0\geq \zeta .c_2(X)=\pi ^*(\xi ). (\pi ^*c_2(Y)+\pi ^*C-\pi ^*c_1(Y).F)=\xi .c_2(Y)+\xi .C.
\end{eqnarray*}
Hence, $\xi .c_2(Y)\leq -\xi .C\leq 0$. 

Since $Y$ satisfies Property A, it follows that $\xi =0$. Consequently $\zeta =\pi ^*(\xi )=0$. 

{\bf Subcase 2.2}: $\alpha >0$. In this case we will obtain a contradiction. 

We use the idea in part  e) of the proof of Theorem 2 in \cite{truong}.  We have a SES of vector bundles over $C$:
\begin{eqnarray*}
0\rightarrow T_C\rightarrow T_Y|_C\rightarrow N_{C/Y}\rightarrow 0.
\end{eqnarray*}

From this, it follows that 
\begin{eqnarray*}
c_1(N_{C/Y})=c_1(Y).C+2g-2=\gamma .
\end{eqnarray*}

Recall that $F$ is the exceptional divisor of the blowup $\pi$. Then $F=\mathbb{P}(N_{C/X})\rightarrow C$ is a ruled surface over $C$. Hence, (see Proposition 2.8 in Chapter 5 in \cite{hartshorne}) there is a line bundle $\mathcal{M}$ over $C$ such that $\mathcal{E} =N_{C/Y}\otimes \mathcal{M}$ is normalized, in the sense that $H^0(\mathcal{E} )\not= 0$, but for every line bundle $\mathcal{L}$ with $c_1(\mathcal{L})<0$ then $H^0(\mathcal{E}\otimes \mathcal{L})=0$. 

Let $f$ be a fiber of the fibration $F\rightarrow C$. Then, (see Proposition 2.9 in Chapter 5 in \cite{hartshorne}), there is  a so-called zero section $C_0\subset F$ with the following properties: 
\begin{eqnarray*}
\tau :=C_0.C_0&=&c_1(\mathcal{E}),\\
C_0.f&=&1.
\end{eqnarray*}

Because $N_{C/Y}$ is decomposable, $\mathcal{E}$ is also decomposable. By part a) of Theorem 2.12 in Section 5 in \cite{hartshorne}, $c_1(\mathcal{E})\leq 0$. Moreover, from
\begin{eqnarray*}
c_1(\mathcal{E})=c_1(\mathcal{N_{C/Y}})+2c_1(\mathcal{M})=c_1(Y).C+2g-2+2c_1(\mathcal{M}),
\end{eqnarray*}
and the assumption that $c_1(Y).C$ is an odd number, we get that $c_1(\mathcal{E})<0$. Hence $\tau <0$.

It can be shown (see for example b) of Lemma 4 in \cite{truong}), that 
\begin{eqnarray*}
C_0=-F.F+\frac{1}{2}(\tau +\gamma )f.
\end{eqnarray*}

By the results in  \cite{truong} (for example d) of the proof of Theorem 2) therein), from $\alpha >0$ and $\zeta ^2=0$ we have
\begin{eqnarray*}
\xi .C=\frac{1}{2}\alpha \gamma .
\end{eqnarray*}

Now we obtain the desired contradiction. Since $\zeta $ is nef and $C_0$ is an effective curve, we have $\zeta .C_0\geq 0$. Hence,
\begin{eqnarray*}
 0&\leq& (\pi ^*(\xi )-\alpha F).(-F.F+\frac{1}{2}(\tau +\gamma )f\\
 &=&\xi .\pi _*(-F.F)+\alpha F.F.F-\frac{1}{2}\alpha (\tau +\gamma )F.f\\
 &=&\xi .C-\alpha \gamma+\frac{1}{2}\alpha (\tau +\gamma )=\frac{\alpha \tau }{2}<0.
\end{eqnarray*}
In the above we used that $\pi _*(-F.F)=\pi _*(C_0)=C$ (see for example Lemma 4 in \cite{truong}), $F.f=-1$, $F.F.F=-\gamma$, $\xi .C=\alpha \gamma /2$ , $\alpha >0$ and $\tau =C_0.C_0< 0$.
\end{proof}

\begin{proof}[Proof of Theorem \ref{Theorem2}]
1) Let $\zeta $ be a nef class on $X_2$ such that $\zeta ^2=0$, $\zeta .c_1(X)=0$ and $\zeta .c_1(X_2)^2\leq 0$. We need to show that $\zeta \in \mathbb{R}.H^2_{alg}(X_2,\mathbb{Q})$. More strongly, we will show that $\zeta$ must be $0$.

Let us denote by $F_j$ the exceptional divisor over $D_j$ of the blowup $\pi _2:X_2\rightarrow X_1$.  We denote by $\pi _1:X_1\rightarrow X_0$ the blowup of $C_0$ at the points $p_i$.

We can write $\zeta =\pi _2^*(\xi )-\sum _{j}\alpha _jF_j$, where $\alpha _j\geq 0$ and $\xi$ is a movable class on $X_1$. Since $D_j$ are disjoint, by intersecting the equations $\zeta ^2=\zeta .c_1(X_2)=0$ with $F_j$, we find as in \cite{truong} that either $\alpha _j=0$ or 
\begin{eqnarray*}
\xi .D_j=\alpha _jc_1(X_1).D_j=\alpha _j(2g_j-2).
\end{eqnarray*}
If $\alpha _j=0$ then $$\xi .D_j=\zeta .D_j'\geq 0=\alpha _jc_1(X_1).D_j,$$ where $D_j'\subset F_j$ is a section whose pushforward is $D_j$. If $\alpha _j\not= 0$ then $\xi .D_j=c_1(X_1).D_j$. Therefore,
\begin{eqnarray*}
0\geq \zeta .c_2(X_2)&=&(\pi _2^*(\xi )-\sum _j\alpha _jF_j).(\pi _2^*c_2(X_1)+\sum _j(\pi _2^*D_j-\pi _2^*c_1(X_j).F_j))\\
&=&\xi .c_2(X_1)+\sum _j(\xi .D_j-\alpha _jc_1(X_1).D_j).
\end{eqnarray*}
 Since each term $\xi .D_j-\alpha _jc_1(X_1).D_j$ is non-negative, we find that $\xi .c_2(X_1)\leq 0$. Because $c_2(X_1)=\pi _1^*c_2(X_0)$, we then get that $(\pi _1)_*(\xi ).c_2(X_0)\leq 0$. Because $(\pi _1)_*(\xi )$  is movable in $X_0$, from the assumption on $c_2(X_0)$ we obtain $(\pi _1)_*(\xi )=0$. From this, it easy follows that $\xi$ and then $\zeta $ are $0$.    
 
2) The proof is similar to that of 1). The difference is now that here for each $j$, either $\alpha _j=0$ or $$\xi .D_j-\alpha _jc_1(X_1).D_j=\frac{\alpha _j}{2}[(2g_j-2)-c_1(X_1).D_j].$$
In the first case
\begin{eqnarray*}
\xi .D_j-\alpha _jc_1(X_1).D_j=\xi .D_j=\zeta .D_j'\geq 0,
\end{eqnarray*}
where $D_j'\subset F_j$ is a section. In the second case, by the assumption $(2g_j-2)-c_1(X_1).D_j\geq 0$, we also have $\xi .D_j-\alpha _jc_1(X_1).D_j\geq 0$. 

Hence,
\begin{eqnarray*}
0\geq -\sum _j(\xi .D_j-\alpha _jc_1(X_1).D_j)\geq \xi .c_2(X_1).
\end{eqnarray*}
Then we can proceed as before. 
 \end{proof}
 \begin{proof}[Proof of Theorem \ref{Theorem3}]
1) For the proof, it suffices to show that for any non-zero nef $\zeta$ on $X$ then either $\zeta .c_1(X)^2\not= 0$ or $\zeta .c_2(X)\not= 0$.

We let $E_1,\ldots ,E_n$ be the exceptional divisors of the blowup $\pi _1:X_1\rightarrow X_0=\mathbb{P}^3$. Let  $F_{i,j}$ be the exceptional divisors of the blowup $\pi _2:X=X_2\rightarrow X_1$. Then we can write 
\begin{eqnarray*}
\zeta &=&\pi _2^*(\xi )-\sum _{i<j}\alpha _{i,j}F_{i,j},\\
\xi &=&\pi _1^*(u)-\sum _{l}\beta _lE_l.
\end{eqnarray*}
Here $u$ is nef on $\mathbb{P}^3$ and $\alpha _{i,j},\beta _l\geq 0$. 

For the proof of 1), it then suffices to show that $\deg (u)=0$.  From 
\begin{eqnarray*}
c_2(X)=\pi _2^*c_2(X_1)+\sum _{i<j}\pi _2^*D_{i,j}-\sum _{i<j}\pi _2^*c_1(X_1).F_{i,j},
\end{eqnarray*}
and the fact that $c_1(X_1).D_{i,j}=0$, the condition $\zeta .c_2(X)=0$ becomes
$\xi .c_2(X_1)+\sum _{i<j}\xi .D_{i,j}=0$. Since $c_2(X_1)=\pi _1^*(c_2(\mathbb{P}^3))$, it follows that $\xi .c_2(X_1)=16 \deg (u)$. We also have that $\xi .D_{i,j}=\deg (u)-\beta _i-\beta _j$ for every $i<j$. Therefore, we obtain
\begin{eqnarray*}
6\deg (u)&=&-\sum _{i<j}\xi .D_{i,j},\\
(6+\frac{n(n-1)}{2})\deg (u)&=&(n-1)\sum _l\beta _l.
\end{eqnarray*}

From the condition $\zeta .c_1(X)^2=0$, we obtain
\begin{eqnarray*}
0&=&\zeta .c_1(X)^2=(\pi _2^*(\xi )-\sum _{i<j}\alpha _{i,j}F_{i,j}).(\pi _2^*c_1(X_1)^2-2\sum _{i<j}\pi _2^*c_1(X_1).F_{i,j}+\sum _{i<j}F_{i,j}^2)\\
&=&\xi . c_1(X_1)^2-\sum _{i<j}\xi .D_{i,j}-2\sum _{i<j}\alpha _{i,j}c_1(X_1).D_{i,j}+\sum _{i<j}\alpha _{i,j}(c_1(X_1).D_{i,j}+2g_{i,j}-2)\\
&=&22 \deg (u)-4\sum _l\beta _l+\sum _{i<j}\alpha _{i,j}(2g_{i,j}-2-c_1(X_1).D_{i,j})\\
&=&22\deg (u)-4\sum _{l}\beta _l-2\sum _{i<j}\alpha _{i,j}.
\end{eqnarray*}
In the above, $g_{i,j}=0$ is the genus of $C_{i,j}$, and $c_1(X_1).D_{i,j}=0$ for all $i<j$. In particular, we obtain
\begin{equation}
\frac{11}{2}\deg (u)\geq \sum _{l}\beta _l=(\frac{6}{n-1}+\frac{n}{2})\deg (u).
\label{Equation3}\end{equation}
From the above inequality, we will finish showing that $\deg (u)=0$. We consider several cases:

Case 1: $n\geq 10$. From Equation (\ref{Equation3}), it follows immediately that $\deg (u)=0$ as wanted. 

Case 2: $6\leq n\leq 9$. In this case, for each $6$ points $p_{i_1},\ldots ,p_{i_6}$ among $n$ points $p_1,\ldots ,p_n$, there is a unique rational normal curve $C\subset \mathbb{P}^3$ of degree $3$ passing through the $6$ chosen points. Let $D\subset X_1$ be the  strict transform of $C$. Then $D$ is different from the curves $D_{i,j}$. Therefore $\pi _2^*D$ is an effective curve, and hence $$3\deg (u)-\sum _{l=1}^6\beta _{i_l}\geq \xi .D=\zeta .\pi _2^*(D)\geq 0.$$  Summing over all such choices of $p_{i_1},\ldots ,p_{i_n}$ we find that
\begin{eqnarray*}
\frac{n}{2}\deg (u)\geq \sum _{l}\beta _l.
\end{eqnarray*}  
Combining this with $$\sum _l\beta _l=(\frac{6}{n-1}+\frac{n}{2})\deg (u),$$
we obtain $\deg (u)=0$.

Case 3: $n=4,5$. In this case, we use rational normal curves to obtain
\begin{eqnarray*}
\frac{n}{3}\deg (u)\geq \sum _l\beta _l.
\end{eqnarray*}
Combining this with 
\begin{eqnarray*}
\sum _{l}\beta _l=(\frac{6}{n-1}+\frac{n}{2})\deg (u),
\end{eqnarray*}
we obtain $\deg (u)=0$.

Case 4: $n=1,2,3$. In this case we have $n\deg (u)\geq \sum _l\beta _l$. Combining this with 
\begin{eqnarray*}
\sum _{l}\beta _l=(\frac{6}{n-1}+\frac{n}{2})\deg (u),
\end{eqnarray*}
we obtain $\deg (u)=0$.
  \end{proof}

\end{document}